\documentclass{article}
\usepackage [utf8] {inputenc}
\usepackage{mathtools}
\usepackage{amsthm}
\usepackage{amssymb}
\usepackage{enumerate}
\usepackage{float}
\usepackage{url}

\newtheorem{theorem}{Theorem}
\newtheorem{proposition}{Proposition}
\newtheorem{lemma}{Lemma}

\DeclareMathOperator{\Sym}{Sym}
\DeclareMathOperator{\Aut}{Aut}
\DeclareMathOperator{\AGL}{A\Gamma L}
\DeclareMathOperator{\GamL}{\Gamma L}
\DeclareMathOperator{\GamO}{\Gamma O}
\DeclareMathOperator{\GL}{GL}
\DeclareMathOperator{\SL}{SL}

\DeclareMathOperator{\SU}{SU}
\DeclareMathOperator{\Sz}{Sz}
\DeclareMathOperator{\GF}{GF}
\DeclareMathOperator{\VO}{VO}
\DeclareMathOperator{\VSz}{VSz}
\DeclareMathOperator{\VD}{VD}
\DeclareMathOperator{\rad}{rad}

\title{On 2-closures of rank~3 groups}
\author{S.V.~Skresanov\thanks{
The work was partially supported by the RFBR grant No.~18-01-00752, and by Mathematical
Center in Akademgorodok under agreement No.~075-15-2019-1613 with the Ministry of Science
and Higher Education of the Russian Federation.
}}
\date{}

\begin{document}

\maketitle
\begin{abstract}
	A permutation group \( G \) on \( \Omega \) is called a \textit{rank}~3 group
	if it has precisely three orbits in its induced action on \( \Omega \times \Omega \).
	The largest permutation group on \( \Omega \) having the same orbits as \( G \) on
	\( \Omega \times \Omega \) is called the 2-\textit{closure} of \( G \).
	A description of 2-closures of rank~3 groups is given. As a special case, it is proved
	that 2-closure of a primitive one-dimensional affine rank~3 permutation group of sufficiently large degree
	is also affine and one-dimensional.
\end{abstract}

\section{Introduction}

Let \( G \) be a permutation group on a finite set \( \Omega \). Recall that the
\textit{rank} of \( G \) is the number of orbits in the induced action of \( G \)
on~\( \Omega \times \Omega \); these orbits are called 2-\textit{orbits}.
If a rank~3 group has even order, then its non-diagonal 2-orbit induces a strongly regular graph
on \( \Omega \), which is called a \textit{rank~\( 3 \) graph}. It is readily seen that a
rank~3 group acts on the corresponding rank~3 graph as an automorphism group.
Notice that an arc-transitive strongly regular graph need not be a rank~3 graph,
since its automorphism group might be intransitive on non-arcs.

Related to this is the notion of a 2-\textit{closure} of a permutation group.
The group \( G^{(2)} \) is the 2-closure of a permutation group \( G \),
if \( G^{(2)} \) is the largest permutation group having the same 2-orbits as \( G \).
Clearly \( G \leq G^{(2)} \), the 2-closure of \( G^{(2)} \) is again \( G^{(2)} \),
and \( G^{(2)} \) has the same rank as \( G \).
Note also that given a rank~3 graph \( \Gamma \) corresponding to the rank~3 group \( G \),
we have \( \Aut(\Gamma) = G^{(2)} \).

The rank~3 groups are completely classified. A primitive rank~3 permutation group either
stabilizes a nontrivial product decomposition, is almost simple or is an affine group.
The rank~3 groups stabilizing a nontrivial product decomposition are given by the classification
of the 2-transitive almost simple groups, see Theorem~4.1~(ii)(a) and \S 5 in~\cite{cameronFinsimp}.
Almost simple rank~3 groups were determined in~\cite{bannai} when the socle is an alternating group,
in~\cite{kantor} when the socle is a classical group
and in~\cite{liebeckSaxl} when the socle is an exceptional or sporadic group. The classification of affine
rank~3 groups was completed in~\cite{liebeckAffine}.

In order to describe the 2-closures of rank~3 permutation groups (or, equivalently, the automorphism groups of rank~3 graphs),
it is essential to know which rank~3 permutation groups give rise to isomorphic graphs.
Despite all the rank~3 groups being known, it is not a trivial task (considerable progress in this direction
was obtained in~\cite{srgw}). In this work we give a detailed description of 2-closures of rank~3 groups.

\begin{theorem}\label{class}
	Let \( G \) be a rank~\( 3 \) permutation group on a set \( \Omega \)
	and suppose that \( |\Omega| \) is sufficiently large.
	Then exactly one of the following is true.
	\begin{enumerate}[(i)]
		\item \( G \) is imprimitive, i.e.\ it preserves a nontrivial decomposition
			\( \Omega = \Delta \times X \).
			Then \( G^{(2)} = \Sym(\Delta) \wr \Sym(X) \).
		\item \( G \) is primitive and preserves a product decomposition
			\( \Omega = \Delta^2 \). Then \( G^{(2)} = \Sym(\Delta) \uparrow \Sym(2) \).
		\item \( G \) is primitive almost simple with socle \( L \), i.e.\ \( L \unlhd G \leq \Aut(L) \).
			Then \( G^{(2)} = N_{\Sym(\Omega)}(L) \), and \( G^{(2)} \) is almost simple with socle \( L \).
		\item \( G \) is a primitive affine group, i.e.\ \( G \leq \AGL_a(q) \) for some \( a \geq 1 \)
			and a prime power~\( q \), moreover, \( G \) does not stabilize a product decomposition.
			Set \( F = \GF(q) \). Then \( G^{(2)} \) is also an affine group and exactly one of the following holds.
			\begin{enumerate}
				\item \( G \leq \AGL_1(q) \). Then \( G^{(2)} \leq \AGL_1(q) \).
				\item \( G \leq \AGL_{2m}(q) \) preserves the bilinear forms graph \( H_q(2, m) \), \( m \geq 3 \) and
					\[ G^{(2)} = F^{2m} \rtimes ((\GL_2(q) \circ \GL_m(q)) \rtimes \Aut(F)). \]
				\item \( G \leq \AGL_{2m}(q) \) preserves the affine polar graph \( \VO_{2m}^{\epsilon}(q) \),
					\( m \geq 2 \), \( \epsilon = \pm \). Then
					\[ G^{(2)} = F^{2m} \rtimes \operatorname{\Gamma O}^{\epsilon}_{2m}(q). \]
				\item \( G \leq \AGL_{10}(q) \) preserves the alternating forms graph \( A(5, q) \). Then
					\[ G^{(2)} = F^{10} \rtimes ((\GamL_5(q) / \{ \pm 1 \}) \times (F^{\times} / (F^{\times})^2)). \]
				\item \( G \leq \AGL_{16}(q) \) preserves the affine half spin graph \( \VD_{5,5}(q) \).
					Then \( G^{(2)} \leq \AGL_{16}(q) \) and we have
					\[ G^{(2)} = F^{16} \rtimes ((F^{\times} \circ \operatorname{Inndiag}(D_5(q))) \rtimes \Aut(F)). \]
				\item \( G \leq \AGL_4(q) \) preserves the Suzuki-Tits ovoid graph \( \VSz(q) \),
					\( q = 2^{2e+1} \), \( e \geq 1 \). Then
					\[ G^{(2)} = F^4 \rtimes ((F^{\times} \times \Sz(q)) \rtimes \Aut(F)). \]
			\end{enumerate}
	\end{enumerate}
\end{theorem}

Up arrow symbol in~(ii) denotes the primitive wreath product (see Section~\ref{nonaffsec}),
notation for graphs in the affine case is explained in Section~\ref{affcase}. We also remark that the value
of \( a \) in (iv) of the theorem is not necessarily minimal subject to \( G \leq \AGL_a(q) \), since it
is not completely defined by the corresponding rank~3 graph and may depend on the group-theoretical structure of \( G \).
Minimal values of \( a \) can be found in Table~\ref{subtab} of Appendix.

It should be noted that the phrase ``sufficiently large'' in this paper means ``larger than some
absolute constant''. Quite often a statement is true but for a finite number of cases (see, for instance,
Lemma~\ref{praeger2}), and since we are generally not interested in these exceptions, we require degrees of
our groups to be ``sufficiently large''.

The proof of Theorem~\ref{class} can be divided into three parts. First we reduce the study to the case when \( G^{(2)} \)
has the same socle as \( G \), and deal with cases~(i)--(iii) (Proposition~\ref{nonabsoc}).
In the affine case~(iv) we apply the classification of affine rank~3 groups~\cite{liebeckAffine},
and compare subdegrees of groups from various classes (Lemma~\ref{a1subdgr} and Proposition~\ref{subdinter});
that allows us to deal with case (a). Finally, we invoke known results on automorphisms of
some families of strongly regular graphs to cover cases (b)--(d), while cases (e) and (f) are treated separately.

Part (iv), (a) of Theorem~\ref{class} can be formulated as a standalone result
that may be of the independent interest.

\begin{theorem}\label{a1inv}
	Let \( G \) be a primitive affine permutation group of rank~3
	and suppose that \( G \leq \AGL_1(q) \) for some prime power \( q \).
	Then either \( G^{(2)} \) lies in \( \AGL_1(q) \), or degree and
	the smallest subdegree of \( G \) are listed in Table~\ref{a1exctab}.
\end{theorem}

It is important to stress that there are cases when a one-dimensional rank~3 group has a nonsolvable
2-closure; see~\cite{skresanovCex} for such an example of degree~\( 2^6 \).

The main motivation for the present study is the application of Theorem~\ref{class}
to the computational 2-closure problem. Namely, the problem asks if given generators of a rank~3 group
one can find generators of its 2-closure in polynomial time. This task influenced the scope of
this work considerably, for instance, while one can find the structure of the normalizer in Theorem~\ref{class}~(iii)
depending on the type of the corresponding rank~3 graph, it is not required for the computational problem
as this normalizer can be computed in polynomial time~\cite{polynorm}.
In other cases it is possible to work directly with associated rank~3 graphs (for example, with Hamming graphs),
but in many situations a more detailed study of relevant groups is required.
The author plans to turn to this problem in his future work.

The author would like to express his gratitude to professors M.~Grechkoseeva,
I.~Ponomarenko and A.~Vasil'ev for numerous helpful comments and suggestions.

\section{Reduction to affine case}\label{nonaffsec}

We will prove Theorem~\ref{class} by dealing with rank~3 groups on a case by case basis.
Recall the following well-known general classification of rank~3 groups.

\begin{proposition}\label{rank3class}
	Let \( G \) be a permutation group of rank~\( 3 \) with socle~\( L \).
	Then \( G \) is transitive and one of the following holds:
	\begin{enumerate}[(i)]
		\item \( G \) is imprimitive,
		\item \( L \) is a direct product of two isomorphic simple groups,
			and \( G \) preserves a nontrivial product decomposition,
		\item \( L \) is nonabelian simple,
		\item \( L \) is elementary abelian.
	\end{enumerate}
\end{proposition}
\begin{proof}
	Transitivity part is clear.
	If \( G \) is primitive, Theorem~4.1 and Proposition~5.1 from \cite{cameronFinsimp}
	imply that it belongs to one of the last three cases from the statement.
\end{proof}


Suppose that \( G \leq \Sym(\Omega) \). Observe that \( G \) acts imprimitively on \( \Omega \)
if and only if the action domain can be identified with a nontrivial Cartesian product \( \Omega = \Delta \times X \),
\( |\Delta| > 1 \), \( |X| > 1 \), where \( G \) permutes blocks of the form \( \Delta \times \{ x \} \),~\( x \in X \).
Denote by \( \Sym(\Delta) \wr \Sym(X) \leq \Sym(\Omega) \) the wreath product of \( \Sym(\Delta) \) and \( \Sym(X) \)
in the imprimitive action, so \( G \leq \Sym(\Delta) \wr \Sym(X) \).

\begin{proposition}\label{imprimCase}
	Let \( G \) be an imprimitive rank~\( 3 \) group on \( \Omega \).
	Let \( \Delta \) be a nontrivial block of imprimitivity of \( G \),
	so \( \Omega \) can be identified with \( \Delta \times X \) for some
	set \( X \). Then \( G^{(2)} = \Sym(\Delta) \wr \Sym(X) \).
\end{proposition}
\begin{proof}
	Set \( H = \Sym(\Delta) \wr \Sym(X) \), where \( \Delta \times X \)
	is identified with \( \Omega \) as in the statement of the proposition. Then \( G \leq H \)
	and since \( G \) and \( H \) are both groups of rank~\( 3 \), we have \( G^{(2)} = H^{(2)} \).
	By \cite[Lemma~2.5]{kalouj} (see also \cite[Proposition~3.1]{evdokimovOdd}), we have 
	\[ (\Sym(\Delta) \wr \Sym(X))^{(2)} = \Sym(\Delta)^{(2)} \wr \Sym(X)^{(2)} = \Sym(\Delta) \wr \Sym(X), \]
	so \( H \) is \( 2 \)-closed. Hence \( G^{(2)} = H^{(2)} = H \), as claimed.
\end{proof}

Suppose that the action domain is a Cartesian power of some set: \( \Omega = \Delta^m \), \( m \geq 2 \) and \( |\Delta| > 1 \).
Denote by \( \Sym(\Delta) \uparrow \Sym(m) \) the wreath product of \( \Sym(\Delta) \) and \( \Sym(m) \)
in the product action, i.e.\ the base group acts on \( \Delta^m \) coordinatewise, while \( \Sym(m) \)
permutes the coordinates. We say that \( G \leq \Sym(\Omega) \)
preserves a nontrivial product decomposition \( \Omega = \Delta^m \)
if \( {G \leq \Sym(\Delta) \uparrow \Sym(m)} \).

If \( G \) preserves a nontrivial product decomposition \( \Omega = \Delta^m \), then \( G \)
induces a permutation group \( G_0 \leq \Sym(\Delta) \). Recall that we can identify \( G \)
with a subgroup of \( G_0 \uparrow \Sym(m) \). We need the following folklore formula
for the rank of a primitive wreath product.

\begin{lemma}\label{rkwr}
	Let \( G \) be a transitive permutation group of rank \( r \).
	Then \( G \uparrow \Sym(m) \) has rank \( \binom{r + m - 1}{m} \).
\end{lemma}
\begin{proof}
	Let \( G \leq \Sym(\Delta) \), and recall that \( \Gamma = G \uparrow \Sym(m) \)
	acts on \( \Delta^m \). Choose \( \alpha_1 \in \Delta \) and set
	\( \overline{\alpha_1} = (\alpha_1, \dots, \alpha_1) \in \Delta^m \).
	Let \( \alpha_1, \dots, \alpha_r \) be representatives
	of orbits of \( G_{\alpha_1} \) on \( \Delta \). 
	Since the point stabilizer \( \Gamma_{\overline{\alpha_1}} \) is equal to \( G_{\alpha_1} \uparrow \Sym(m) \),
	points \( (\alpha_{i_1}, \dots, \alpha_{i_m}) \), where \( 1 \leq i_1 \leq \dots \leq i_m \leq r \),
	form a set of representatives of orbits of \( \Gamma_{\overline{\alpha_1}} \) on \( \Delta^m \).
	The number of indices \( i_1, \dots, i_m \) satisfying \( {1 \leq i_1 \leq \dots \leq i_m \leq r} \)
	is equal to the number of weak compositions of \( m \) into \( r \) parts, hence the claim is proved.
\end{proof}

Observe that in the particular case when \( \Omega = \Delta^2 \), the wreath product \( \Sym(\Delta) \uparrow \Sym(2) \)
has rank~3 and one of its 2-orbits can be viewed as edges of the Hamming graph \( H(2, |\Delta|) \).

\begin{proposition}\label{prodCase}
	Let \( G \) be a primitive rank~\( 3 \) permutation group on \( \Omega \)
	preserving a nontrivial product decomposition \( \Omega = \Delta^m \), \( m \geq 2 \).
	Then \( m = 2 \), a \( 2 \)-orbit of \( G \) induces a Hamming graph and \( G^{(2)} = \Sym(\Delta) \uparrow \Sym(2) \).
\end{proposition}
\begin{proof}
	Let \( H \) denote \( \Sym(\Delta) \uparrow \Sym(m) \), and recall that by Lemma~\ref{rkwr},
	\( H \)	has rank~\( {{2+m-1}\choose m} = m+1 \) as a permutation group.
	Since \( G \leq H \), we have \( m+1 \leq 3 \).
	Therefore \( m = 2 \) and \( H \) is a rank~\( 3 \) group. Then \( G^{(2)} = H^{(2)} \)
	and it suffices to show that \( H \) is \( 2 \)-closed.

	A 2-orbit of \( H \) induces the Hamming graph \( H(2, q) \) on \( \Omega \), where \( q = |\Delta| \).
	By \cite[Theorem~9.2.1]{brouwerDRG}, \( \Aut(H(2, q)) = \Sym(q) \uparrow \Sym(2) \). It readily follows
	that \( H^{(2)} = \Aut(H(2, q)) = H \), completing the proof.
\end{proof}

In order to find 2-closures in the last two cases of Proposition~\ref{rank3class},
we need to show that 2-closure almost always preserves the socle of a rank~3 group.

\begin{lemma}\label{praeger2}
	Let \( G \) be a primitive rank~\( 3 \) permutation group
	and suppose that \( G \) and \( G^{(2)} \) have different socles.
	Then either \( G \) preserves a nontrivial product decomposition, or \( G \)
	belongs to a finite set of almost simple groups.
\end{lemma}
\begin{proof}
	From \cite[Theorem~2]{praegerClosure} it follows that either \( G \) preserves a nontrivial
	product decomposition, or \( G \) and \( G^{(2)} \) are almost simple groups with different socles.
	The latter situation applies only to a finite number of rank~3 groups, by \cite[Theorem~1]{liebeck2Closure},
	so the claim follows.
\end{proof}

\begin{lemma}\label{asprod}
	Let \( G \) be a primitive rank~\( 3 \) permutation group with nonabelian simple socle.
	Then \( G \) does not preserve a nontrivial product decomposition.
\end{lemma}
\begin{proof}
	Suppose this is not the case and \( G \) is a primitive rank~3 permutation group on \( \Omega \)
	preserving a nontrivial product decomposition and having a nonabelian simple socle \( L \).
	Then \cite[Theorem~8.21]{praegerDecomp} implies that either \( L \) is \( A_6 \) and \( |\Omega| = 36 \),
	or \( L = M_{12} \) and \( |\Omega| = 144 \), or \( L = \operatorname{Sp}_4(q) \), \( q \geq 4 \), \( q \) even
	and \( |\Omega| = q^4(q^2-1)^2 \). One can easily check that neither of these situations occurs in rank~3
	by inspecting the classification of almost simple rank~3 groups.
	The reader is referred to~\cite[Table~5]{buekenhout} for alternating socles, \cite[Table~9]{buekenhout} for sporadic socles
	and \cite[Tables~6 and 7]{buekenhout} for classical socles.
\end{proof}

It should be noted that an almost simple group with rank larger than~3 might preserve a nontrivial
product decomposition, see~\cite[Section~1.3]{praegerDecomp}.

\begin{proposition}\label{almsimp}
	Let \( G \) be a primitive rank~\( 3 \) permutation group on \( \Omega \)
	with nonabelian simple socle \( L \). Then apart from a finite number of exceptions,
	\( G^{(2)} \) has socle \( L \) and \( G^{(2)} = N_{\Sym(\Omega)}(L) \).
\end{proposition}
\begin{proof}
	By Lemma~\ref{asprod}, \( G \) does not preserve a nontrivial product decomposition,
	hence by Lemma~\ref{praeger2}, apart from finitely many exceptions 2-closure \( G^{(2)} \) has the same socle as \( G \).
	Set \( N = N_{\Sym(\Omega)}(L) \). Clearly \( G^{(2)} \leq N \), and to establish equality it suffices to show
	that \( N \) is a rank~3 group. 

	Suppose that this not the case and \( N \) is 2-transitive.
	By \cite[Proposition~5.2]{cameronFinsimp}, \( N \) has a unique minimal normal subgroup,
	and since \( L \) is a minimal normal subgroup of \( N \), the socle of \( N \) must be equal to \( L \).
	Hence \( N \) is an almost simple 2-transitive group with socle \( L \).

	The possibilities for a socle of a 2-transitive almost simple group are all known
	and moreover, apart from finitely many cases such a socle is a 2-transitive group itself
	(see Theorem~5.3~(S) and the following notes in~\cite{cameronFinsimp}). 
	Since \( L \leq G \), and \( G \) is not 2-transitive, we yield a contradiction.
	Therefore \( N \) is a rank~3 group and \( G^{(2)} = N \).
\end{proof}

We summarize the results of this section in the following.
\begin{proposition}\label{nonabsoc}
	Let \( G \) be a rank~3 permutation group on \( \Omega \).
	Then apart from a finite number of exceptions exactly one of the following holds.
	\begin{enumerate}[(i)]
		\item \( G \) is imprimitive, i.e.\ it preserves a nontrivial decomposition
			\( \Omega = \Delta \times X \).
			Then \( G^{(2)} = \Sym(\Delta) \wr \Sym(X) \).
		\item \( G \) is primitive and preserves a product decomposition
			\( \Omega = \Delta^2 \). Then \( G^{(2)} = \Sym(\Delta) \uparrow \Sym(2) \).
		\item \( G \) is a primitive almost simple group with socle \( L \), i.e.\ \( L \unlhd G \leq \Aut(L) \).
			Then \( G^{(2)} = N_{\Sym(\Omega)}(L) \), and \( G^{(2)} \) is almost simple with socle \( L \).
		\item \( G \) is a primitive affine group which does not stabilize a product decomposition.
			Then \( G^{(2)} \) is also an affine group.
	\end{enumerate}
\end{proposition}

\section{Affine case}\label{affcase}

In the previous section we reduced the task of describing 2-closures of rank~3 groups
to the case when the group in question is affine. Recall that a primitive permutation
group \( G \leq \Sym(\Omega) \) is called \textit{affine}, if it has a unique minimal normal subgroup \( V \)
equal to its socle, such that \( V \) is an elementary abelian \( p \)-group for some prime \( p \) and 
\( G = V \rtimes G_0 \) for some \( G_0 < G \). The permutation domain \( \Omega \)
can be identified with \( V \) in such a way that \( V \) acts on it by translations, and \( G_0 \)
acts on it as a subgroup of \( \GL(V) \). Clearly \( G_0 \) is the stabilizer of the zero vector in \( V \)
under such identification.

If \( G_0 \) acts semilinearly on \( V \) as a \( \GF(q) \)-vector space, where \( q \) is a power of~\( p \),
then we write \( G_0 \leq \GamL_m(q) \), where \( \GamL_m(q) \) is the full semilinear group
and \( V \simeq \GF(q)^m \). If the field is clear from the context, we may use \( \GamL(V) = \GamL_m(q) \) instead.
We write \( \AGL_m(q) \) for the full affine semilinear group.

Now we are ready to state the classification of affine rank~3 groups.

\begin{theorem}[\cite{liebeckAffine}]\label{affrankthree}
	Let \( G \) be a finite primitive affine permutation group of rank~\( 3 \)
	and of degree \( n = p^d \), with socle \( V \), where \( V \simeq \GF(p)^d \)
	for some prime~\( p \), and let \( G_0 \) be the stabilizer of the zero vector in \( V \).
	Then \( G_0 \) belongs to one of the following classes.
	\begin{enumerate}
		\item[\textbf{(A)}] Infinite classes. These are:
			\begin{enumerate}[(1)]
				\item \( G_0 \leq \GamL_1(p^d) \);
				\item \( G_0 \) is imprimitive as a linear group;
				\item \( G_0 \) stabilizes the decomposition of \( V \simeq \GF(q)^{2m} \)
					into \( V = V_1 \otimes V_2 \), where \( p^d = q^{2m} \), \( \dim V_1 = 2 \) and \( \dim V_2 = m \);
				\item \( G_0 \unrhd \SL_m(\sqrt{q}) \) and \( p^d = q^m \), where \( 2 \) divides \( \frac{d}{m} \);
				\item \( G_0 \unrhd \SL_2(\sqrt[3]{q}) \) and \( p^d = q^2 \), where \( 3 \) divides \( \frac{d}{2} \);
				\item \( G_0 \unrhd \SU_m(q) \) and \( p^d = q^{2m} \);
				\item \( G_0 \unrhd \Omega_{2m}^{\pm}(q) \) and \( p^d = q^{2m} \);
				\item \( G_0 \unrhd \SL_5(q) \) and \( p^d = q^{10} \);
				\item \( G_0 \unrhd B_3(q) \) and \( p^d = q^8 \);
				\item \( G_0 \unrhd D_5(q) \) and \( p^d = q^{16} \);
				\item \( G_0 \unrhd \operatorname{Sz}(q) \) and \( p^d = q^4 \).
			\end{enumerate}
		\item[\textbf{(B)}] `Extraspecial' classes.
		\item[\textbf{(C)}] `Exceptional' classes.
	\end{enumerate}
	Moreover, classes (B) and (C) constitute only a finite number of groups.
\end{theorem}

Observe that the only case when a primitive affine rank~3 group can lie in some other class
from the statement of Proposition~\ref{nonabsoc} is when it preserves a nontrivial product decomposition.
This is precisely case (A2) of the classification, so this situation can indeed happen.

Recall that each rank~3 group gives rise to a rank~3 graph.
By~\cite[Table~11.4]{srgw}, groups from case (A) of the theorem correspond to the following series of graphs:
\begin{itemize}
	\item One-dimensional affine graphs (i.e.\ arising from case (A1)).
		These graphs are either Van Lint--Schrijver, Paley or Peisert graphs~\cite{muzychukOneDim};
	\item Hamming graphs. These graphs correspond to linearly imprimitive groups;
	\item Bilinear forms graph \( H_q(2, m) \), where \( m \geq 2 \) and \( q \) is a prime power.
		These graphs correspond to groups fixing a nontrivial tensor decomposition;
	\item Affine polar graph \( \VO_{2m}^{\epsilon}(q) \), where \( m \geq 2 \), \( \epsilon = \pm \) and \( q \) is a prime power;
	\item Alternating forms graph \( A(5, q) \), where \( q \) is a prime power;
	\item Affine half spin graph \( \VD_{5,5}(q) \), where \( q \) is a prime power;
	\item Suzuki-Tits ovoid graph \( \VSz(q) \), where \( q = 2^{2e+1} \), \( e \geq 1 \).
\end{itemize}
The reader is referred to~\cite{srgw} for the construction and basic properties of mentioned graphs.

It should be noted that different cases of Theorem~\ref{affrankthree} may correspond to isomorphic graphs.
Table~\ref{graphtab} lists affine rank~3 groups from case (A) and indicates respective rank~3 graphs.
In Tables~\ref{subtab} and~\ref{a1subtab} we provide degrees and subdegrees of affine rank~3 groups in case~(A).
Relevant tables and some comments on sources of data used are collected in Appendix.

Our first goal is to show that almost all pairs of affine rank~3 graphs can be distinguished based on their subdegrees.
We start with the class~(A1). The following lemma summarizes some of the arithmetical conditions
for the subdegrees of the corresponding groups.
\begin{lemma}\label{a1subdgrdescr}
	Let \( G \) be a primitive affine rank~3 group from class (A1) having degree
	\( n = p^d \), where \( p \) is a prime.
	Denote by \( m_1, m_2 \) the subdegrees of \( G \) and suppose that \( m_1 < m_2 \).
	Then \( m_1 \) divides \( m_2 \) and \( \frac{m_2}{m_1} \) divides \( d \).
\end{lemma}
\begin{proof}
	See \cite[Proposition~3.3]{foulserRank3} for the first claim and
	\cite[Theorem~3.7, (4)]{foulserRank3} for the second.
\end{proof}

The following lemmas apply conditions from Lemma~\ref{a1subdgrdescr} to groups from classes (B), (C) and (A).
\begin{lemma}\label{a1inb}
	Let \( G \) be a primitive affine rank~3 group from class (B).
	Suppose that \( G \) has the same subdegrees as a group from class (A1).
	Then the degree and subdgrees of \( G \) are one of the following:
	\( (7^2, 24, 24) \), \( (17^2, 96, 192) \), \( (23^2, 264, 264) \),
	\( (3^6, 104, 624) \), \( (47^2, 1104, 1104) \), \( (3^4, 16, 64) \), \( (7^4, 480, 1920) \).
\end{lemma}
\begin{proof}
	Let \( n \) denote the degree of \( G \), and let \( m_1 \leq m_2 \) be the subdegrees.
	In Table~\ref{casebtab} all possible subdegrees of groups from class (B) are listed.
	We apply Lemma~\ref{a1subdgrdescr}. For instance, if \( n = 29^2 \) then \( m_1 = 168 \),
	\( m_2 = 672 \). The quotient \( \frac{m_2}{m_1} = 4 \) does not divide \( 2 \), hence
	this case cannot happen. Other cases are treated in the same manner.
\end{proof}

\begin{lemma}\label{a1inc}
	Let \( G \) be a primitive affine rank~3 group from class (C).
	Suppose that \( G \) has the same subdegrees as a group from class (A1).
	Then the degree and subdegrees of \( G \) are one of the following:
	\( (3^4, 40, 40) \), \( (2^{12}, 315, 3780) \), \( (89^2, 2640, 5280) \).
\end{lemma}
\begin{proof}
	Follows from Lemma~\ref{a1subdgrdescr} and Table~\ref{casectab}.
\end{proof}

\begin{lemma}\label{a1subdgr}
	Let \( G \) be a primitive affine rank~3 group from class (A) and
	suppose that \( G \) has the same subdegrees as a group from class (A1).
	Then either \( G \) lies in (A1) or degree and subdegrees of \( G \) are one of the following:
	\( (3^2, 4, 4) \), \( (3^4, 16, 64) \), \( (3^6, 104, 624) \),
	\( (2^4, 5, 10) \), \( (2^6, 21, 42) \), \( (2^8, 51, 204) \), \((2^{10}, 93, 930)\),\\ \((2^{12}, 315, 3780)\), \( (2^{16}, 3855, 61680) \),
	\( (5^2, 8, 16) \).
\end{lemma}
\begin{proof}
	Suppose that \( G \) does not lie in class (A1), but shares subdegrees with some group from (A1).
	Notice that in cases (A3) through (A11), exactly one of the subdegrees is divisible by \( p \), so the subdegrees
	are not equal. In case (A2) subdegrees are the same if and only if \( p^m = 3 \), and consequentially \( n = 9 \).
	This situation is the first example in our list of parameters, hence from now on we may assume that \( n > 9 \),
	and subdegrees of \( G \) are not equal.

	Let \( m_1 \) and \( m_2 \) denote the subdegrees of \( G \), where, as shown earlier, we may assume \( m_1 < m_2 \).
	Since \( m_1 \) and \( m_1 \) are subdegrees of some group from the class (A1), Lemma~\ref{a1subdgrdescr}
	yields that \( m_1 \) divides \( m_2 \) and the number \( u = \frac{m_2}{m_1} \) divides~\( d \), where \( n = p^d \).

	Now, since \( G \) belongs to one of the classes (A2)--(A11), we apply the above
	arithmetical conditions in each case. We consider some classes together, since
	they give rise to isomorphic rank~3 graphs and hence have the same formulae for subdegrees.
	The reader is referred to Table~\ref{subtab} for the list of subdegrees in question.

	\begin{itemize}
		\item[\textbf{(A2)}] \( u = \frac{p^m-1}{2} \) and since \( u \) divides \( d = 2m \),
			we have \( p^m - 1 \leq 4m \). It follows that \( (n, m_1, m_2) \) is one of
			\( (3^2, 4, 4) \), \( (3^4, 16, 64) \) or \( (5^2, 8, 16) \).
		\item[\textbf{(A3)--(A5)}] We write \( r \) for the highest power of \( p \) dividing \( m_2 \),
			so the second subdegree is equal to \( r(r^m-1)(r^{m-1}-1) \) for some \( m \geq 2 \).
			
			We have \( u = r\frac{r^{m-1}-1}{r+1} \) and hence \( u \geq \frac{r^{m-1}-1}{2} \).
			Now \( r^{2m} = p^d \geq p^{\frac{r^{m-1}-1}{2}} \). Using inequalities \( m \geq 2 \) and \( p \geq 2 \),
			we obtain \( 2r^{8(m-1)} \geq 2^{r^{m-1}} \). Therefore \( r^{m-1} \leq 44 \) and an exhaustive search
			yields the following possibilities for \( n \), \( m_1 \), \( m_2 \):
			\( (2^6, 21, 42) \), \((2^{10}, 93, 930) \), \( (2^{12}, 315, 3780) \), \( (3^6, 104, 624) \).
		\item[\textbf{(A6), (A7)}] \( u = q^{m-1}\frac{q-1}{q^{m-1} \pm 1} \). Numbers \( q^{m-1} \)
			and \( q^{m-1} \pm 1 \) are coprime, so \( q^{m-1} \pm 1 \) divides \( q-1 \).
			That is possible only when \( m = 2 \), so we have \( u = q \). Now \( 2^q \leq p^q \leq p^d = q^4 \),
			so \( q \leq 16 \). Hence we have the following possibilities for \( n \), \( m_1 \), \( m_2 \) in this case:
			\( (2^4, 5, 10) \), \( (2^8, 51, 204) \), \( (2^{16}, 3855, 61680) \).
		\item[\textbf{(A8)}] \( u = q^3 - q^2 \frac{q+1}{q^2+1} \). Since \( q^2 \) and \( q^2+1 \)
			are coprime, \( q^2 + 1 \) must divide \( q+1 \). This can not happen, so this case does not occur.
		\item[\textbf{(A9)}] \( u = q^3 \frac{q-1}{q^3+1} \). Since \( q^3 + 1 \) does not divide \( q-1 \),
			this case does not occur.
		\item[\textbf{(A10)}] \( u = q^5 - q^3 \frac{q^2+1}{q^3+1} \). Since \( q^3+1 \) does not divide \( q^2+1 \),
			this case does not occur.
		\item[\textbf{(A11)}] \( u = q \) and \( p^d = q^4 \). Hence we obtain the same
			possible parameters as in cases (A6), (A7).
	\end{itemize}

	In all cases considered we either got a contradiction or got one of the possible exceptions recorded in the
	statement. The claim is proved.
\end{proof}

As an immediate corollary we derive that 2-closures of primitive rank~3 subgroups
of \( \AGL_1(q) \) also lie in \( \AGL_1(q) \) (Theorem~\ref{a1inv}), apart from a finite number of exceptions.
\medskip

\noindent
\emph{Proof of Theorem~\ref{a1inv}.}
	Suppose that \( G \) and \( G^{(2)} \) have different socles.
	Since \( G \) is not almost simple, Lemma~\ref{praeger2} implies that
	\( G^{(2)} \) and thus \( G \) must preserve a nontrivial product decomposition.
	In that situation \( G \) has subdegrees of the form \( 2(\sqrt{n}-1) \), \( (\sqrt{n}-1)^2 \),
	in particular, \( G \) has subdegrees as a group from class (A2) and hence parameters of \( G \)
	are listed in Lemma~\ref{a1subdgr}.
	We may assume that \( G \) does not preserve a nontrivial product decomposition and
	so \( G \) and \( G^{(2)} \) have equal socles.
	The claim now follows from Theorem~\ref{affrankthree} and Lemmas~\ref{a1inb}--\ref{a1subdgr}.\qed
\medskip

Note that Lemmas~\ref{a1inb}--\ref{a1subdgr} list degrees and subdegrees of possible exceptions to Theorem~\ref{a1inv};
in Table~\ref{a1exctab} of Appendix we collect these data in one place.

\medskip

Now we move on to establish a partial analogue of Lemma~\ref{a1subdgr} for classes (A2)--(A11).
First we need to recall some notions related to quadratic and bilinear forms.

Let \( V \) be a vector space over a field \( F \). Given a symmetric bilinear form \( f : V \times V \to F \),
the radical of \( f \) is \( \rad(f) = \{ x \in V \mid f(x, y) = 0 \text{ for all } y\} \);
we say that \( f \) is \emph{non-singular}, if \( \rad(f) = 0 \).
If \( \kappa : V \times V \to F \) is a quadratic form with an associated bilinear form \( f \), then
the radical of \( \kappa \) is \( \rad(\kappa) = \rad(f) \cap \{ x \in V \mid \kappa(x) = 0 \} \).
We say that \( \kappa \) is \emph{non-singular}, if \( \rad(\kappa) = 0 \), and we say that
\( \kappa \) is \emph{non-degenerate}, if \( \rad(f) = 0 \).

If \( F \) has odd characteristic, then \( \rad(\kappa) = \rad(f) \).
If \( F \) has even characteristic and \( \kappa \) is non-singular, then the dimension of \( \rad(f) \)
is at most one, \( f \) induces a non-singular alternating form on \( V/\rad(f) \) and,
hence, the dimension of \( V/\rad(f) \) is even (see~\cite[Section~3.4.7]{wilson}).
Therefore if the dimension of \( V \) is even, then the notions of non-singular and non-degenerate quadratic forms
coincide regardless of the characteristic.

Now we can describe the construction of the affine polar graph \( \VO_{2m}^{\epsilon}(q) \), \( m \geq 2 \).
Let \( V \) be a \( 2m \)-dimensional vector space over \( \GF(q) \), and let \( \kappa : V \to \GF(q) \)
be a non-singular quadratic form of type \( \epsilon \). Vertices of the graph \( \VO_{2m}^{\epsilon}(q) \)
are identified with vectors from \( V \), and two distinct vertices \( u, v \in V \) are joined by an edge
if \( \kappa(u - v) = 0 \).

Allowing some abuse of terminology, we say that subdegrees of a rank~3 graph are simply subdegrees
of the respective rank~3 group.

\begin{proposition}\label{subdinter}
	If two affine rank~3 graphs of sufficiently large degree have the same subdegrees,
	then they are isomorphic apart from the following exceptions:
	\( \VSz(q) \) and \( \VO_4^-(q) \) for \( q = 2^{2e+1} \), \( e \geq 1 \),
	or Paley and Peisert graphs. In particular, graphs \( H_q(2, 2) \) and \( \VO_4^+(q) \) are isomorphic.
\end{proposition}
\begin{proof}
	By Lemma~\ref{a1subdgr}, if one of the graphs in question arises from the case (A1), then the second
	graph also comes from (A1). By Table~\ref{a1subtab}, Van Lint-Schrijver graph has unequal subdegrees,
	while Paley and Peisert graphs have equal subdegrees, hence in this case graphs are either isomorphic
	or it is a Paley graph and a Peisert graph. We may now assume that our graphs are not one-dimensional.
	
	Notice that given \( n = p^d \) for \( p \) prime, the largest subdegree of graphs from classes
	(A3)--(A11) is divisible by \( p \), while this is not the case in class (A2) (we may assume that
	\( n > 9 \), subdegrees are not equal in this case). This settles the claim for class (A2).

	We compare subdegrees of other classes and collect the relevant information in Table~\ref{casestab}. Let us explain the procedure
	in the case \( H_q(2, m) \) vs. \( \VO^{\pm}_{2\overline{m}}(\overline{q}) \) only, since other cases are treated
	similarly.

	Consider the graph \( H_q(2, m) \). The number of its vertices is equal to \( n = q^{2m} \)
	and the second subdegree is equal to \( q(q^m - 1)(q^{m-1} - 1) \). Recall that \( n = p^d \)
	for some prime \( p \), and the largest power of \( p \) dividing the second subdegree is~\( q \).
	In the case of the graph \( \VO^{\epsilon}_{2\overline{m}}(\overline{q}) \), we have \( n = \overline{q}^{2\overline{m}} \)
	and the largest power of \( p \) dividing the second subdegree is \( \overline{q}^{\overline{m}-1} \).
	We obtain a system of equations
	\[ q^{2m} = \overline{q}^{2\overline{m}}, \, q = \overline{q}^{\overline{m}-1},\]
	which is written in the relevant cell of Table~\ref{casestab}. We derive that \( m = \frac{\overline{m}}{\overline{m}-1} \),
	and hence \( m = \overline{m} = 2 \), \( q = \overline{q} \).
	Now, the second subdegree for \( \VO^{\epsilon}_4(q) \) is \( q(q-1)(q^2+(-1)^\epsilon) \).
	Therefore \( \epsilon = + \), which gives us the first example of affine rank~3 graphs with same subdegrees.
	Other cases are dealt with in the same way.

	Now, Table~\ref{casestab} lists two cases when graphs from different classes have the same subdegrees,
	namely, \( H_q(2, 2) \), \( \VO_4^+(q) \) and \( \VSz(q) \), \( \VO_4^-(q) \).
	To finish the proof of the proposition, we show that graphs \( H_q(2, 2) \) and \( \VO_4^+(q) \) are in fact isomorphic.

	Identify vertices of \( H_q(2, 2) \) with \( 2 \times 2 \) matrices over \( \GF(q) \),
	and recall that two vertices are connected by an edge if the rank of their difference is~1.
	A nonzero \( 2 \times 2 \) matrix has rank~1 precisely when its determinant is zero:
	\[
		\operatorname{rk}\begin{pmatrix}u_1 & u_3\\u_4 & u_2\end{pmatrix} = 1 \iff u_1u_2 - u_3u_4 = 0.
	\]
	It can be easily seen that \( u_1u_2 - u_3u_4 \) is a non-degenerate quadratic form on~\( \GF(q)^4 \),
	so \( H_q(2, 2) \) is isomorphic to the affine polar graph \( \VO_4^\epsilon(q) \). By comparing subdegrees
	we derive that \( \epsilon = + \), and we are done.
\end{proof}

It should be noted that \( \VSz(q) \) and \( \VO_4^-(q) \) in fact have the
same parameters as strongly regular graphs (see~\cite[Table~24]{buekenhout}).
In Lemma~\ref{szvo} we will see that these graphs are actually
not isomorphic since they have non-isomorphic automorphism groups.

Paley and Peisert graphs are generally not isomorphic (see~\cite{peisert}), but have the same
parameters since they are strongly regular and self-complementary (i.e.\ isomorphic to their complements).
\medskip

Recall that in order to describe 2-closures of rank~3 groups it suffices to find full automorphism
groups of corresponding rank~3 graphs. Hamming graphs were dealt with in Proposition~\ref{prodCase},
and graphs arising in the case (A1) were covered in Theorem~\ref{a1inv}. We are left with five cases:
bilinear forms graph, affine polar graph, alternating forms graph, affine half spin graph and
the Suzuki-Tits ovoid graph. In most of these cases the full automorphism group was described
earlier in some form, and we state relevant results here.

For two groups \( G_1 \) and \( G_2 \) let \( G_1 \circ G_2 \) denote their central product.
Note that the central product \( \GL(U) \circ \GL(W) \) has a natural action on the
tensor product \( U \otimes W \).

\begin{proposition}[{\cite[Theorem~9.5.1]{brouwerDRG}}]\label{bilaut}
	Let \( q \) be a prime power and \( m \geq 2 \). Set \( G = \Aut(H_q(2, m)) \)
	and \( F = \GF(q) \).
	If \( m > 2 \), then
	\[ G = F^{2m} \rtimes ((\GL_2(q) \circ \GL_m(q)) \rtimes \Aut(F)). \]
	If \( m = 2 \), then
	\[ G = F^4 \rtimes (((\GL_2(q) \circ \GL_2(q)) \rtimes \Aut(F)) \rtimes C_2), \]
	where the additional automorphism of order~2 exchanges components of simple tensors.
\end{proposition}

Let \( V \) be a vector space endowed with a quadratic form \( \kappa \).
We say that a vector \( v \in V \) is \emph{isotropic} if \( \kappa(v) = 0 \).

\begin{lemma}[\cite{schroder}]\label{alexlest}
	Let \( V \) be a vector space over some (possibly finite) field \( F \), and suppose that \( \dim V \geq 3 \).
	Let \( \kappa : V \to F \) be a non-singular quadratic form, possessing an isotropic vector. If \( f \) is a permutation
	of \( V \) with the property that
	\[ \kappa(x - y) = 0 \Leftrightarrow \kappa(x^f - y^f) = 0, \]
	then \( f \in \AGL(V) \) and \( f : x \mapsto x^\phi + v \), \( v \in V \),
	where \( \phi \in \GamL(V) \) is a semisimilarity of \( \kappa \), i.e.\ there exist \( \lambda \in F^{\times} \)
	and \( \alpha \in \Aut(F) \) such that \( \kappa(x^\phi) = \lambda \kappa(x)^{\alpha} \) for all \( x \in V \).
\end{lemma}

Denote by \( \operatorname{\Gamma O}^{\epsilon}_{2m}(q) \) the group of all semisimilarities of a non-degenerate
quadratic form of type \( \epsilon \) on the vector space of dimension \( 2m \) over the finite field of order \( q \).
The reader is referred to \cite[Sections~2.7 and~2.8]{kleidman} for the structure and properties
of groups \( \operatorname{\Gamma O}^{\epsilon}_{2m}(q) \).

\begin{proposition}\label{affpolaut}
	Let \( q \) be a prime power and \( m \geq 2 \). Set \( F = \GF(q) \). Then
	\[ \Aut(\VO^{\epsilon}_{2m}(q)) = F^{2m} \rtimes \operatorname{\Gamma O}^{\epsilon}_{2m}(q), \, \epsilon = \pm. \]
\end{proposition}
\begin{proof}
	Recall that the graph \( \VO^{\epsilon}_{2m}(q) \) is defined by a vector space
	\( V = F^{2m} \) over \( F \) and a non-singular (or, equivalently, non-degenerate)
	quadratic form \( \kappa : V \to F \). Since \( m \geq 2 \), we have \( \dim V \geq 3 \)
	and \( \kappa \) possesses an isotropic vector. The claim now follows from Lemma~\ref{alexlest}.
\end{proof}

\begin{proposition}[{\cite[Theorem~9.5.3]{brouwerDRG}}]\label{altaut}
	Let \( q \) be a prime power and set \( F = \GF(q) \). Then
	\[ \Aut(A(5, q)) = F^{10} \rtimes ((\GamL_5(q) / \{ \pm 1 \}) \times (F^{\times} / (F^{\times})^2)). \]
\end{proposition}

\begin{lemma}\label{vd55norm}
	Let \( q \) be a prime power, \( q^{16} = p^d \), let \( F = \GF(q) \), \( V = F^{16} \) and set \( G = \Aut(\VD_{5,5}(q)) \).
	Then \( G = V \rtimes G_0 \), and
	\[ F^{\times} \circ D_5(q) \leq G_0 = N_{\GL_d(p)}(D_5(q)), \]
	where \( D_5(q) \) acts on the spin module.
	Moreover, \( G_0/F^{\times} \) is an almost simple group and \( G_0 \leq \GamL_{16}(q) \).
\end{lemma}
\begin{proof}
	Set \( H = V \rtimes (F^{\times} \circ D_5(q)) \). By \cite[Lemma~2.9]{liebeckAffine}, \( D_5(q) \) has
	two orbits on the set of lines \( P_1(V) \), so \( H \) is an affine rank~3 group of type (A10).
	Clearly \( G = H^{(2)} \) so by Lemma~\ref{praeger2}, \( G \) is an affine rank~3 group.
	By Proposition~\ref{subdinter}, \( G \) belongs to class (A10) and the main result of \cite{liebeckAffine}
	implies that \( G_0 \leq N_{\GL_d(p)}(D_5(q)) \). By \cite[(1.4)]{liebeckAffine}, the generalized Fitting subgroup
	of \( G_0/F^{\times} \) is simple, hence this quotient group is almost simple.
	By Hering's theorem~\cite{hering} (see also~\cite[Appendix~1]{liebeckAffine}),
	the normalizer \( N_{\GL_d(p)}(D_5(q)) \) cannot be transitive on nonzero vectors of \( V \),
	so \( G_0 = N_{\GL_d(p)}(D_5(q)) \) as claimed.

	Finally, let \( a \) be the minimal integer such that \( G_0 \leq \GamL_a(p^{d/a}) \).
	By Table~\ref{subtab}, \( a = 16 \), so the last inclusion follows.
\end{proof}

Denote by \( D_5(q) \) an orthogonal group of universal type, in particular, recall that
\( |Z(D_5(q))| = \gcd(4, q^5-1) \) (see~\cite[Table~5]{atlas}).
We write \( \operatorname{Inndiag}(D_5(q)) \) for the overgroup of \( D_5(q) \) in \( \Aut(D_5(q)) \),
containing all diagonal automorphisms.

\begin{proposition}\label{vd55aut}
	Let \( q \) be a prime power, and set \( F = \GF(q) \). Then
	\[ \Aut(\VD_{5,5}(q)) = F^{16} \rtimes ((F^{\times} \circ \operatorname{Inndiag}(D_5(q))) \rtimes \Aut(F)). \]
\end{proposition}
\begin{proof}
	We follow~\cite[Lemma~2.9]{liebeckAffine}. Take \( K = E_6(q) \) to be of universal type,
	so that \( |Z(K)| = \gcd(3, q-1) \). The Dynkin diagram of \( K \) is:
	\[
		\overset{\mathclap{\alpha_1}}{\circ} -
		\overset{\mathclap{\alpha_3}}{\circ} -
		\overset{\mathclap{\alpha_4}}{
				\underset{\underset{\textstyle\circ\mathrlap{\scriptstyle\,\alpha_2}}{\textstyle\vert}}
				{\circ}} -
		\overset{\mathclap{\alpha_5}}{\circ} -
		\overset{\mathclap{\alpha_6}}{\circ}
	\]
	Let \( \Sigma \) be the set of roots and let \( x_{\alpha}(t) \), \( h_{\alpha}(t) \) be Chevalley generators of \( K \).
	Write \( X_\alpha = \{ x_\alpha(t) | t \in F \} \). Let \( P \) be a parabolic subgroup of \( K \) corresponding to the set of roots
	\( \{ \alpha_2, \alpha_3, \alpha_4, \alpha_5, \alpha_6 \} \), and let \( P = UL \) be its Levi decomposition. Moreover,
	\( L = MH \) and we may choose \( P \) such that
	\[ U = \langle X_\alpha \mid \alpha \in \Sigma^+,\; \alpha \text{ involves } \alpha_1 \rangle, \]
	\[ M = \langle X_{\pm\alpha_i} \mid 2 \leq i \leq 6 \rangle, \]
	where \( M \) is of universal type and \( H = \langle h_{\alpha_i}(t) | t \in F, \; 1 \leq i \leq 6 \rangle \) is the Cartan subgroup.
	In \cite[Lemma~2.9]{liebeckAffine} it was shown that \( M \simeq D_5(q) \),
	the group \( U \) is elementary abelian of order \( q^{16} \) and in fact,
	it is a spin module for \( M \). By~\cite[Theorem~2.6.5~(f)]{cfsg3}, \( H \) induces diagonal automorphisms on \( M \),
	and by~\cite[Section~1, B]{minperm} it induces the full group of diagonal automorphisms.
	Recall that for an element \( h \) of \( H \) we have \( x_{\alpha}(t)^h = x_{\alpha}(k\cdot t) \) for some \( k \in F \).
	In particular, diagonal automorphisms of \( D_5(q) \) commute with the action of the field \( F \) on \( U \).

	Let \( \phi \) be a generator of the field automorphisms group of \( K \),
	and note that one can identify that group with \( \Aut(F) \);
	in particular, \( \phi \) acts on \( F \) under such an identification.
	By \cite[Theorem~2.5.1~(c)]{cfsg3}, generators \( x_{\alpha}(t) \) and \( h_{\alpha}(t) \) are carried to
	\( x_{\alpha}(t^\phi) \) and \( h_{\alpha}(t^\phi) \) by \( \phi \), so field automorphisms normalize \( U \), \( M \) and~\( H \).
	Furthermore, \( \phi \) induces the full group of field automorphisms on \( M \).

	Set \( T = L \rtimes \langle \phi \rangle \). We have \( M \unlhd T \) and \( T \) induces
	all field and diagonal automorphisms on \( M \).
	Set \( \overline{T} = T/Z(K) \) and \( \overline{M} = MZ(K)/Z(K) \).
	By \cite[Theorem~2.6.5~(e)]{cfsg3}, the centralizer \( C_{\Aut(K)}(U) \) is the image of \( Z(U) \) in \( \Aut(K) \).
	Therefore \( \overline{T} \) acts faithfully on \( U \), and since \( |Z(M)| \) is coprime to \( |Z(K)| \),
	we derive that \( \overline{M} \simeq M \simeq D_5(q) \).
	Hence we have an embedding \( \overline{T} \leq \GL_d(p) \), where \( |U| = p^d \),
	and, with some abuse of notation, \( \overline{T} \leq N_{\GL_d(p)}(D_5(q)) \). By Lemma~\ref{vd55norm}, the latter
	normalizer is an almost simple group (modulo scalars), and thus we have shown that it contains all field and diagonal automorphisms
	of \( D_5(q) \). It is left to show that it does not contain graph automorphisms.

	Suppose that a graph automorphism \( \psi \) lies in \( G_0 = N_{\GL_d(p)}(D_5(q)) \), and recall that
	\( \overline{M} \simeq D_5(q) \). By \cite[Lemma~2.9]{liebeckAffine}, there is an orbit \( \Delta \) of \( G_0 \)
	on nonzero vectors of \( U \), such that the point stabilizer \( \overline{M}_\delta \), \( \delta \in \Delta \)
	is a parabolic subgroup of type \( A_4 \).
	Since \( \psi \) preserves the orbit \( \Delta \) and normalizes \( \overline{M} \),
	it must take a point stabilizer \( \overline{M}_\delta \)
	to the point stabilizer \( \overline{M}_{\delta'} \) for some \( \delta' \in \Delta \), in particular,
	it takes \( \overline{M}_\delta \) to a conjugate subgroup. 
	That is impossible, since by~\cite[Theorem~2.6.5~(c)]{cfsg3}, automorphism \( \psi \) interchanges conjugacy classes of
	parabolic subgroups of type \( A_4 \), so the final claim is proved.
\end{proof}

Recall the construction of the graph \( \VSz(q) \), \( q = 2^{2e+1} \), \( e \geq 1 \).
Set \( F = \GF(q) \), \( V = F^4 \) and let \( \sigma \) be an automorphism of \( F \)
acting as \( \sigma(x) = x^{2^{e+1}} \). Define the subset \( O \) of the projective space
\( P_1(V) \) by
\[ O = \{ (0, 0, 1, 0) \} \cap \{ (x, y, z, 1) \mid z = xy + x^2x^\sigma + y^\sigma\}, \]
where vectors are written projectively. The vertex set of \( \VSz(q) \) is \( V \) and two
vectors are connected by an edge, if a line connecting them has a direction in~\( O \).

Recall that \( \Sz(q) \leq \GL_4(q) \) is faithfully represented on \( P_1(V) \)
and induces the group of all collineations which preserve the Suzuki-Tits ovoid \( O \)	(see~\cite[Chapter XI, Theorem~3.3]{huppert3}).
Clearly scalar transformations preserve the preimage of \( O \) in \( V \), and it can be easily
seen that \( O^\alpha = O \) for any \( \alpha \in \Aut(F) \). Hence the following group
\[ H = V \rtimes ((F^{\times} \times \Sz(q)) \rtimes \Aut(F)) \]
acts as a group of automorphisms of \( \VSz(q) \). By~\cite[Lemma~16.4.6]{hirschfeld}, \( \Sz(q) \)
acts transitively on \( P_1(V) \setminus O \), hence \( H \) is a rank~3 group.

We will show that \( H \) is the full automorphism group of \( \VSz(q) \), but first we need
to note the following basic fact.

\begin{lemma}\label{szvo}
	If \( q = 2^{2e+1} \), \( e \geq 1 \), then there is no subgroup of \( \Aut(\VO_4^-(q)) \)
	isomorphic to \( \Sz(q) \). In particular, graphs \( \VO_4^-(q) \) and \( \VSz(q) \) are not isomorphic.
\end{lemma}
\begin{proof}
	Suppose the contrary, so that \( \Sz(q) \) is a subgroup of \( \Aut(\VO_4^-(q)) \).
	By Proposition~\ref{affpolaut}, we have \( \Aut(\VO_4^-(q)) \simeq V \rtimes \GamO_4^-(q) \)
	for some elementary abelian group \( V \). Recall that the orthogonal group \( \Omega_4^-(q) \)
	is a normal subgroup of~\( \GamO_4^-(q) \), and the quotient \( \GamO_4^-(q)/\Omega_4^-(q) \) is solvable.
	Clearly \( V \) is also solvable, and since \( \Sz(q) \) is simple, we obtain an embedding of
	\( \Sz(q) \) into \( \Omega_4^-(q) \). Yet that is impossible, as can be easily seen by inspection
	of maximal subgroups of \( \Omega_4^-(q) \), see, for instance,~\cite[Table~8.17]{bray}.
	That is a contradiction, so the first claim is proved.

	The second claim follows from the fact that \( \Sz(q) \) lies in \( \Aut(\VSz(q)) \).
\end{proof}

\begin{proposition}\label{szaut}
	Let \( q = 2^{2e+1} \), where \( e \geq 1 \), and set \( F = \GF(q) \). Then
	\[ \Aut(\VSz(q)) = F^4 \rtimes ((F^{\times} \times \Sz(q)) \rtimes \Aut(F)). \]
\end{proposition}
\begin{proof}
	Let \( H = F^4 \rtimes ((F^{\times} \times \Sz(q)) \rtimes \Aut(F)) \)
	be a rank~3 group acting on \( \VSz(q) \) by automorphisms. Set \( G = \Aut(\VSz(q)) \)
	and recall that \( G = H^{(2)} \). By Lemma~\ref{praeger2}, \( G \) is an affine group with the same socle as \( H \),
	and by Proposition~\ref{subdinter} and Table~\ref{graphtab}, it follows
	that \( G \) lies in class (A7) or (A11). By Lemma~\ref{szvo}, the first possibility does not happen,
	so \( G \) is a group from class (A11). Denote by \( H_0 \) and \( G_0 \) zero stabilizers in \( H \) and \( G \)
	respectively. Notice that \( H_0 \leq G_0 \).

	By Theorem~\ref{affrankthree} and Table~\ref{subtab}, we have \( G_0 \leq \GamL_4(q) \)
	and \( \Sz(q) \unlhd G_0 \). By~\cite[(1.4)]{liebeckAffine}, given \( Z = Z(\GL_4(q)) \simeq F^\times \),
	the generalized Fitting subgroup of \( G_0/(G_0 \cap Z) \) is a simple group. Hence \( G_0/(G_0 \cap Z) \)
	is an almost simple group with socle~\( \Sz(q) \).

	The outer automorphisms group of \( \Sz(q) \) consists of field automorphisms only (see~\cite[Table~5]{atlas}), so
	\[ |G_0| \leq |Z|\cdot|\Aut(\Sz(q))| \leq |F^\times||\Sz(q)||\Aut(F)|. \]
	Since \( H_0 \simeq (F^{\times} \times \Sz(q)) \rtimes \Aut(F) \), the order of \( H_0 \) coincides
	with the value on right hand side of the inequality. Now \( H_0 = G_0 \) and the claim is proved.
\end{proof}

\emph{Proof of Theorem~\ref{class}.} Let \( G \) be a rank~3 group of sufficiently large degree.
By Proposition~\ref{nonabsoc}, we may assume that \( G \) is a primitive affine group which does not
stabilize a product decomposition and, moreover, \( G^{(2)} \) is also an affine group.
By Theorem~\ref{affrankthree}, \( G \) is either a one-dimensional affine group (class (A1)), or
preserves a bilinear forms graph \( H_q(2, m) \), \( m \geq 2 \), an affine polar graph
\( \VO_{2m}^{\epsilon}(q) \), \( \epsilon = \pm \), \( m \geq 2 \), alternating forms graph \( A(5, q) \),
affine half-spin graph \( \VD_{5,5}(q) \) or Suzuki-Tits ovoid graph \( \VSz(q) \).

The full automorphism groups of these graphs (i.e.\ 2-closures of respective groups) are described in
Theorem~\ref{a1inv} (one-dimensional affine groups), Proposition~\ref{bilaut} (bilinear forms graph),
Proposition~\ref{affpolaut} (affine-polar graph), Proposition~\ref{altaut} (alternating forms graph),
Proposition~\ref{vd55aut} (affine half-spin graph) and Proposition~\ref{szaut} (Suzuki-Tits ovoid graph).

Since by Proposition~\ref{subdinter}, the graph \( H_q(2, 2) \) is isomorphic to \( \VO_4^+(q) \),
we may exclude it from the bilinear forms case. Now it is easy to see that cases
considered in Theorem~\ref{class}~(iv) are mutually exclusive. Indeed, it suffices to prove
that graphs from different cases are not isomorphic. By Proposition~\ref{subdinter},
if two affine rank~3 graphs have the same subdegrees, then they belong to the same case except for
\( \VSz(q) \) and \( \VO_4^-(q) \), \( q = 2^{2e+1} \), \( e \geq 1 \)
(note that we group one-dimensional affine graphs into one case). By Lemma~\ref{szvo},
graphs \( \VSz(q) \) and \( \VO_4^-(q) \) are not isomorphic, which proves the claim.

Finally, inclusions of the form \( G \leq \AGL_a(q) \) can be read off Table~\ref{subtab}.
Notice that in some cases we do not give the minimal value of \( a \), for example, if
\( \SU_m(q) \leq G \) lies in class (A6), then \( G \leq \AGL_m(q^2) \), but we list
the inclusion \( G \leq \AGL_{2m}(q) \). This completes the proof of the theorem. \qed
\bigskip

\newpage
\section{Appendix}\label{appendix}

In this section we collect some relevant tabular data. Table~\ref{subtab} lists information
on affine rank~3 groups from class (A), namely, for each group \( G \) it provides rough group-theoretical
structure (column ``Type of \( G \)''), degree \( n \) and subdegrees. Column ``\( a \)'' gives
the smallest integer \( a \) such that the stabilizer of the zero vector \( G_0 \) lies in \( \GamL_a(p^{d/a}) \).
Most of information in Table~\ref{subtab} is taken from \cite[Table~12]{liebeckAffine}, see also \cite[Table~10]{buekenhout}
for the values of \( a \).

Table~\ref{a1subtab} lists the subdegrees of one-dimensional affine rank~3 groups.
The first column specifies the type of graph associated to the group in question, next two columns
provide degree and subdegrees, and the last column lists additional constraints on parameters involved.
By~\cite{muzychukOneDim}, these graphs turn out to be either Van Lint--Schrijver, Paley or Peisert graphs.
See~\cite[Section~2]{vanlint} for the parameters of the Van Lint--Schrijver graph; parameters
of Paley and Peisert graphs are computed using the fact that these graphs are self-complementary.

Table~\ref{graphtab} lists rank~3 graphs corresponding to rank~3 groups from classes (A1)--(A11),
cf.~\cite[Table~11.4]{srgw}. Terminology and graph notation is mostly consistent with \cite{srgw}, see also \cite[Table~10]{buekenhout}.

Table~\ref{casestab} records information on when some families of affine rank~3 graphs can have identical subdegrees,
the procedure for building this table being described in Proposition~\ref{subdinter}.
Trivial cases (when two graphs are the same) are not listed, also graphs from cases (A1) and (A2) are omitted,
since they are dealt with separately.

Tables~\ref{casebtab} and~\ref{casectab} list degrees and subdegrees of affine rank~3 groups from classes (B) and (C),
without repetitions (i.e.\ parameter sets are listed only once, regardless of whether several groups possess same parameters).
If the smaller subdegree divides the largest, the last column gives the respective quotient; otherwise a dash is placed.
Information in Table~\ref{casebtab} is taken from \cite[Theorem~1.1]{foulserLowRank}
and \cite[Table~13]{liebeckAffine}, see also \cite[Table~11]{buekenhout}. Information in Table~\ref{casectab} before the
horizontal line is taken from \cite[Theorem~5.3]{foulserRank3}, but notice that we exclude the case of \( 119^2 \),
since \( 119 \) is not a prime number (that error was observed by Liebeck in~\cite{liebeckAffine}). Information in Table~\ref{casectab}
after the horizontal line is taken from \cite[Table~14]{liebeckAffine}, with the correction for the case of
\( \operatorname{Alt}(9) \), where subdegrees should be \( 120 \), \( 135 \) instead of \( 105 \), \( 150 \),
as noted in~\cite[Table~12]{buekenhout}.

Table~\ref{a1exctab} lists parameters of possible exceptions to Theorem~\ref{a1inv}.
The table consists of three subtables, corresponding to classes (A), (B) and (C) of Theorem~\ref{affrankthree},
i.e.\ values for the first subtable are taken from Lemma~\ref{a1subdgr},
for the second from Lemma~\ref{a1inb}, and for the third from Lemma~\ref{a1inc}.
Each subtable lists degrees and smallest subdegrees of possible exceptions.
Notice that parameters of one-dimensional affine rank~3 groups stabilizing a nontrivial product decomposition
are collected in the subtable for the class (A).

\begin{table}[h]
\caption{Class (A) in the classification of affine rank 3 groups}\label{subtab}
\begin{tabular}{l l l l c}
	\hline
	       & Type of \( G \)               & \( n = p^d \) & \( a \) & Subdegrees \\
	\hline
	(A1):  & \( G_0 < \Gamma L_1(p^d) \)   & \( p^d \)     & \( 1 \) & See Table~\ref{a1subtab}\\
	(A2):  & \( G_0 \) imprimitive         & \( p^{2m} \)  & \( 2m \) & \( 2(p^m - 1) \), \( (p^m - 1)^2 \)\\
	(A3):  & tensor product                & \( q^{2m} \)  & \( 2m \) & \( (q+1)(q^m - 1) \), \( q(q^m - 1)(q^{m-1} - 1) \)\\
	(A4):  & \( G_0 \unrhd \SL_m(\sqrt{q}) \) & \( q^m \)   & \( m \) &
	       \( (\sqrt{q}+1)(\sqrt{q}^m - 1) \), \( \sqrt{q}(\sqrt{q}^m - 1)(\sqrt{q}^{m-1} - 1) \)\\
	(A5):  & \( G_0 \unrhd \SL_2(\sqrt[3]{q}) \) & \( q^2 \) & \( 2 \) &
	       \( (\sqrt[3]{q}+1)(q - 1) \), \( \sqrt[3]{q}(q - 1)(\sqrt[3]{q}^2 - 1) \)\\
	(A6):  & \( G_0 \unrhd \SU_m(q) \)      & \( q^{2m} \) & \( m \) & \(
	       \begin{cases}
		       (q^m - 1)(q^{m-1} + 1), q^{m-1}(q-1)(q^m - 1), \, m \text{ even}\\
		       (q^m + 1)(q^{m-1} - 1), q^{m-1}(q-1)(q^m + 1), \, m \text{ odd} 
	       \end{cases} \)\\
	(A7):  & \( G_0 \unrhd \Omega^{\epsilon}_{2m}(q) \) & \( q^{2m} \) & \( 2m \) & \(
	       \begin{cases}
		       (q^m - 1)(q^{m-1} + 1), q^{m-1}(q-1)(q^m - 1), \, \epsilon = + \\
		       (q^m + 1)(q^{m-1} - 1), q^{m-1}(q-1)(q^m + 1), \, \epsilon = - 
	       \end{cases} \)\\
	(A8):  & \( G_0 \unrhd \SL_5(q) \)      & \( q^{10} \) & \( 10 \) & \( (q^5 - 1)(q^2 + 1) \), \( q^2(q^5 - 1)(q^3 - 1) \)\\
	(A9):  & \( G_0 \unrhd B_3(q) \)       & \( q^8 \)    & \( 8 \) & \( (q^4 - 1)(q^3 + 1) \), \( q^3(q^4 - 1)(q - 1) \)\\
	(A10): & \( G_0 \unrhd D_5(q) \)       & \( q^{16} \) & \( 16 \) & \( (q^8 - 1)(q^3 + 1) \), \( q^3(q^8 - 1)(q^5 - 1) \)\\
	(A11): & \( G_0 \unrhd \operatorname{Sz}(q) \)        & \( q^4 \)    & \( 4 \) & \( (q^2 + 1)(q - 1) \), \( q(q^2 + 1)(q - 1) \)\\
	\hline
\end{tabular}

\caption{Subdegrees of one-dimensional affine rank~3 groups}\label{a1subtab}
\begin{tabular}{c c c c}
	\hline
	Graph & Degree & Subdegrees & Comments\\
	\hline
	Van Lint--Schrijver & \( q = p^{(e-1)t} \) & \( \frac{1}{e}(q-1) \), \( \frac{1}{e}(e-1)(q-1) \) &
	\( e > 2\) is prime, \( p \) is primitive\(\pmod e \)\\
	Paley & \( q \) & \( \frac{1}{2}(q-1) \), \( \frac{1}{2}(q-1) \) & \( q \equiv 1 \pmod 4 \)\\
	Peisert & \( q = p^{2t} \) & \( \frac{1}{2}(q-1) \), \( \frac{1}{2}(q-1) \) & \( p \equiv 3 \pmod 4 \)\\
	\hline
\end{tabular}
\end{table}

\begin{table}[h]
\caption{Rank 3 graphs in class (A)}\label{graphtab}
\begin{tabular}{l l l p{4cm}}
	\hline
	       & Type of \( G \)               & Graph & Comments \\
	\hline
	(A1):  & \( G_0 < \GamL_1(p^d) \)      & Van Lint--Schrijver, Paley or Peisert graph & \\
	(A2):  & \( G_0 \) imprimitive         & Hamming graph & \\
	(A3):  & tensor product                & bilinear forms graph \( H_q(2, m) \) & \\
	(A4):  & \( G_0 \unrhd \SL_m(\sqrt{q}) \) & bilinear forms graph \( H_{\sqrt{q}}(2, m) \) &
		  \( \SL_m(\sqrt{q}) \) stabilizes an \( m \)-dimensional subspace over \( \GF(\sqrt{q}) \) \\
	(A5):  & \( G_0 \unrhd \SL_2(\sqrt[3]{q}) \) & bilinear forms graph \( H_{\sqrt[3]{q}}(2, 2) \) &
		  \( \SL_2(\sqrt[3]{q}) \) stabilizes a \( 2 \)-dimensional subspace over \( \GF(\sqrt[3]{q}) \)\\
	(A6):  & \( G_0 \unrhd \SU_m(q) \)      & affine polar graph \( \VO^{\epsilon}_{2m}(q) \), \( \epsilon = (-1)^m \) & \\
	(A7):  & \( G_0 \unrhd \Omega^{\epsilon}_{2m}(q) \) & affine polar graph \( \VO^{\epsilon}_{2m}(q) \) & \\
	(A8):  & \( G_0 \unrhd \SL_5(q) \)     & alternating forms graph \( A(5, q) \) & \\
	(A9):  & \( G_0 \unrhd B_3(q) \)       & affine polar graph \( \VO^+_8(q) \) & \\
	(A10): & \( G_0 \unrhd D_5(q) \)       & affine half spin graph \( \VD_{5,5}(q) \) & \\
	(A11): & \( G_0 \unrhd \operatorname{Sz}(q) \) & Suzuki-Tits ovoid graph \( \operatorname{VSz}(q) \) & \\
	\hline
\end{tabular}
\end{table}

\begin{table}[h]
\caption{Intersections between classes based on subdegrees}\label{casestab}
\begin{tabular}{| c | c | c | c | c |}
	\hline
	& \( \VO^{\pm}_{2\overline{m}}(\overline{q}) \) & \( A(5, \overline{q}) \)
	& \( \VD_{5,5}(\overline{q}) \) & \( \operatorname{VSz}(\overline{q}) \) \\
	\hline
	\( H_q(2, m) \) &
	 \begin{tabular}{@{}c@{}}
		 \( q^{2m} = \overline{q}^{2\overline{m}} \)\\
		 \( q = \overline{q}^{\overline{m}-1} \)\\
		 \( m = \frac{\overline{m}}{\overline{m}-1} \)\\
		 \( m = \overline{m} = 2 \), \( q = \overline{q} \)
	 \end{tabular} &
	 \begin{tabular}{@{}c@{}}
		 \( q^{2m} = \overline{q}^{10} \)\\
		 \( q = \overline{q}^2 \)\\
		 \( m = \frac{10}{4} \)\\
		 Impossible
	 \end{tabular} &
	 \begin{tabular}{@{}c@{}}
		 \( q^{2m} = \overline{q}^{16} \)\\
		 \( q = \overline{q}^3 \)\\
		 \( m = \frac{8}{3} \)\\
		 Impossible
	 \end{tabular} &
	 \begin{tabular}{@{}c@{}}
		 \( q^{2m} = \overline{q}^4 \)\\
		 \( q = \overline{q} \)\\
		 \( m = 2 \)\\
		 \( q(q^2-1)(q-1) = q(q^2+1)(q-1) \)\\
		 Impossible
	 \end{tabular}\\
	\hline
	\( \VO^{\pm}_{2m}(q) \) & &
	 \begin{tabular}{@{}c@{}}
		 \( q^{2m} = \overline{q}^{10} \)\\
		 \( q^{m-1} = \overline{q}^2 \)\\
		 \( m = \frac{5}{3} \)\\
		 Impossible
	 \end{tabular} &
	 \begin{tabular}{@{}c@{}}
		 \( q^{2m} = \overline{q}^{16} \)\\
		 \( q^{m-1} = \overline{q}^3 \)\\
		 \( m = \frac{8}{5} \)\\
		 Impossible
	 \end{tabular} &
	 \begin{tabular}{@{}c@{}}
		 \( q^{2m} = \overline{q}^4 \)\\
		 \( q^{m-1} = \overline{q} \)\\
		 \( m = 2 \), \( q = \overline{q} \)
	 \end{tabular}\\
	 \hline
	 \( A(5, q) \) & & &
	 \begin{tabular}{@{}c@{}}
		 \( q^{10} = \overline{q}^{16} \)\\
		 \( q^2 = \overline{q}^3 \)\\
		 Impossible
	 \end{tabular} &
	 \begin{tabular}{@{}c@{}}
		 \( q^{10} = \overline{q}^4 \)\\
		 \( q^2 = \overline{q} \)\\
		 Impossible
	 \end{tabular}\\
	 \hline
	 \( \VD_{5,5}(q) \) & & & &
	 \begin{tabular}{@{}c@{}}
		 \( q^{16} = \overline{q}^4 \)\\
		 \( q^3 = \overline{q} \)\\
		 Impossible
	 \end{tabular}\\
	\hline
\end{tabular}
\end{table}

\begin{table}[h]
\caption{Subdegrees of rank 3 groups in class (B)}\label{casebtab}
\begin{tabular}{l l l}
	\hline
	\( n = p^d \) & Subdegrees \( m_1 \), \( m_2 \) & \( \frac{m_2}{m_1} \) if it is an integer\\
	\( 2^6 \) & \( 27, 36 \) & ---\\
	\( 3^4 \) & \( 32, 48 \) & ---\\
	\( 7^2 \) & \( 24, 24 \) & \( 1 \)\\
	\( 13^2 \) & \( 72, 96 \) & ---\\
	\( 17^2 \) & \( 96, 192 \) & \( 2 \)\\
	\( 19^2 \) & \( 144, 216 \) & ---\\
	\( 23^2 \) & \( 264, 264 \) & \( 1 \)\\
	\( 3^6 \) & \( 104, 624 \) & \( 6 \)\\
	\( 29^2 \) & \( 168, 672 \) & \( 4 \)\\
	\( 31^2 \) & \( 240, 720 \) & \( 3 \)\\
	\( 47^2 \) & \( 1104, 1104 \) & \( 1 \)\\
	\( 3^4 \) & \( 32, 48 \) & ---\\
	\( 3^4 \) & \( 16, 64 \) & \( 4 \)\\
	\( 5^4 \) & \( 240, 384 \) & ---\\
	\( 7^4 \) & \( 480, 1920 \) & \( 4 \)\\
	\( 3^8 \) & \( 1440, 5120 \) & ---\\
	\hline
\end{tabular}
\end{table}

\begin{table}[h]
\caption{Subdegrees of rank 3 groups in class (C)}\label{casectab}
\begin{tabular}{l l l}
	\hline
	\( n = p^d \) & Subdegrees \( m_1 \), \( m_2 \) & \( \frac{m_2}{m_1} \) if it is an integer\\
	\( 3^4 \) & \( 40, 40 \) & \( 1 \)\\
	\( 2^8 \) & \( (2^4-1)\cdot 5, (2^4-1)\cdot 12 \) & ---\\
	\( 5^4 \) & \( (5^2-1)\cdot 6, (5^2-1)\cdot 20 \) & ---\\
	\( 31^2 \) & \( (31-1)\cdot 12, (31-1)\cdot 20 \) & ---\\
	\( 41^2 \) & \( (41-1)\cdot 12, (41-1)\cdot 30 \) & ---\\
	\( 7^4 \) & \( (7^2-1)\cdot 20, (7^2-1)\cdot 30 \) & ---\\
	\( 2^{12} \) & \( (2^6-1)\cdot 5, (2^6-1)\cdot 60 \) & \( 12 \)\\
	\( 71^2 \) & \( (71-1)\cdot 12, (71-1)\cdot 60 \) & \( 5 \)\\
	\( 79^2 \) & \( (79-1)\cdot 20, (79-1)\cdot 60 \) & \( 3 \)\\
	\( 89^2 \) & \( (89-1)\cdot 30, (89-1)\cdot 60 \) & \( 2 \)\\
	\( 5^6 \) & \( (5^3-1)\cdot 6, (5^3-1)\cdot 120 \) & \( 20 \)\\
	\hline\\
	\( 2^6 \) & \( 18, 45 \) & ---\\
	\( 5^4 \) & \( 144, 480 \) & ---\\
	\( 2^8 \) & \( 45, 210 \) & ---\\
	\( 7^4 \) & \( 720, 1680 \) & ---\\
	\( 2^8 \) & \( 120, 135 \) & ---\\
	\( 2^8 \) & \( 102, 153 \) & ---\\
	\( 3^6 \) & \( 224, 504 \) & ---\\
	\( 7^4 \) & \( 240, 2160 \) & \( 9 \)\\
	\( 3^5 \) & \( 22, 220 \) & \( 10 \)\\
	\( 3^5 \) & \( 110, 132 \) & ---\\
	\( 2^{11} \) & \( 276, 1771 \) & ---\\
	\( 2^{11} \) & \( 759, 1288 \) & ---\\
	\( 3^{12} \) & \( 65520, 465920 \) & ---\\
	\( 2^{12} \) & \( 1575, 2520 \) & ---\\
	\( 5^6 \) & \( 7560, 8064 \) & ---\\
	\hline
\end{tabular}
\end{table}

\begin{table}[h]
\caption{Possible exceptions to Theorem~\ref{a1inv}}\label{a1exctab}
	(A)
\begin{tabular}{c|c|c|c|c|c|c|c|c|c|c}
	\hline
	Degree & \(2^4\) & \( 2^6 \) & \( 2^8 \) & \( 2^{10} \) & \( 2^{12} \) & \( 2^{16} \) & \(3^2\) & \(3^4\) & \(3^6\) & \( 5^2 \)\\
	Subdegree & 5 & 21 & 51 & 93 & 315 & 3855 & 4 & 16 & 104 & 8\\
	\hline
\end{tabular}

	(B)
\begin{tabular}{c|c|c|c|c|c|c|c}
	\hline
	Degree & \( 3^4 \) & \( 3^6 \) & \( 7^2 \) & \( 7^4 \) & \( 17^2 \) & \( 23^2 \) & \( 47^2 \)\\
	Subdegree & 16 & 104 & 24 & 480 & 96 & 264 & 1104\\
	\hline
\end{tabular}

	(C)
\begin{tabular}{c|c|c|c}
	\hline
	Degree & \( 2^{12} \) & \( 3^4 \) & \( 89^2 \)\\
	Subdegree & 315 & 40 & 2640\\
	\hline
\end{tabular}
\end{table}

\bigskip

\noindent
{\sl Saveliy V. Skresanov\\
Sobolev Institute of Mathematics, 4 Acad. Koptyug avenue\\
and\\
Novosibirsk State University, 1 Pirogova street,\\
630090 Novosibirsk, Russia\\
e-mail: skresan@math.nsc.ru}


\clearpage
\begin{thebibliography}{10}

\bibitem{bannai}
E.~Bannai.
\newblock {\em Maximal subgroups of low rank of finite symmetric and alternating groups}.
\newblock J. Fac. Sci. Univ. Tokyo, 18:475--486, 1972.

\bibitem{bray}
J.~N. Bray, D.~F. Holt, C.~M. Roney-Dougal.
\newblock {\em The Maximal Subgroups of the Low-Dimensional Finite Classical Groups}.
\newblock LMS Lecture Notes 407, Cambridge Univ. Press, 2013.

\bibitem{brouwerDRG}
A.~E. Brouwer, A.~M. Cohen, A. Neumaier.
\newblock {\em Distance-Regular Graphs}.
\newblock Ergebnisse der Mathematik und ihrer Grenzgebiete. 3. Folge / A Series
  of Modern Surveys in Mathematics. Springer Berlin Heidelberg, 1989.

\bibitem{srgw}
A.~E. Brouwer, H. Van~Maldeghem.
\newblock {\em Strongly regular graphs}.
\newblock \url{https://homepages.cwi.nl/~aeb/math/srg/rk3/srgw.pdf} [preprint]

\bibitem{buekenhout}
F.~Buekenhout, H. Van~Maldeghem.
\newblock {\em A characterization of some rank 2 incidence geometries by their automorphism group.}
\newblock Mitt. Mathem. Sem. Giessen, 218:1--70, 1994.

\bibitem{cameronFinsimp}
P.~J. Cameron.
\newblock {\em Finite permutation groups and finite simple groups.}
\newblock Bulletin of the London Mathematical Society, 13(1):1--22, 1981.

\bibitem{atlas}
J.~H. Conway.
\newblock {\em Atlas of Finite Groups: Maximal Subgroups and Ordinary
  Characters for Simple Groups}.
\newblock Clarendon Press, 1985.

\bibitem{evdokimovOdd}
S. Evdokimov, I. Ponomarenko.
\newblock {\em Two-closure of odd permutation group in polynomial time.}
\newblock Discrete Mathematics, 235(1):221 -- 232, 2001.

\bibitem{foulserRank3}
D.~A. Foulser, M.~J. Kallaher.
\newblock {\em Solvable, flag-transitive, rank 3 collineation groups.}
\newblock Geometriae Dedicata, 7:111--130.

\bibitem{foulserLowRank}
D.~A. Foulser.
\newblock {\em Solvable primitive permutation groups of low rank.}
\newblock Transactions of the American Mathematical Society, 143:1--54, 1969.

\bibitem{cfsg3}
D.~Gorenstein, R.~Lyons, R.~Solomon.
\newblock {\em The Classification of the Finite Simple Groups, Number 3.}
\newblock Mathematical surveys and monographs, 1994.

\bibitem{hering}
C.~Hering.
\newblock {\em Transitive linear groups and linear groups which contain irreducible subgroups of prime order, II.}
\newblock J.~Algebra, 93:151--164, 1985.

\bibitem{hirschfeld}
J.~W.~P. Hirschfeld.
\newblock {\em Finite Projective Spaces of Three Dimensions}.
\newblock Oxford mathematical monographs. Clarendon Press, 1985.

\bibitem{huppert3}
B.~Huppert, N.~Blackburn.
\newblock {\em Finite groups III}.
\newblock Number 243. Springer, 1982.

\bibitem{kalouj}
L.~A.~Kalu\v znin, M.~H.~Klin.
\newblock \emph{On some numerical invariants of permutation groups}.
Latv. Mat. E\v zegod., 18(1):81--99, 1976 [in Russian].

\bibitem{kantor}
W.~M. Kantor, R.~A. Liebler.
\newblock \emph{The rank 3 permutation representations of the finite classical groups}.
Trans. Amer. Math. Soc., 271:1--71, 1982.

\bibitem{kleidman}
P.~B. Kleidman, M.~W. Liebeck.
\newblock {\em The Subgroup Structure of the Finite Classical Groups}.
\newblock London Mathematical Society Lecture Note Series. Cambridge University Press, 1990.

\bibitem{liebeckAffine}
M.~W. Liebeck.
\newblock {\em The affine permutation groups of rank three}.
\newblock Proceedings of the London Mathematical Society, s3-54(3):477--516, 1987.

\bibitem{liebeck2Closure}
M.~W. Liebeck, C.~E. Praeger, J. Saxl.
\newblock {\em On the 2-closures of finite permutation groups.}
\newblock Journal of the London Mathematical Society, s2-37(2):241--252, 1988.

\bibitem{liebeckSaxl}
M.~W. Liebeck, J. Saxl.
\newblock {\em The finite primitive permutation groups of rank three.}
\newblock Bull. London Math. Soc., 18:165--172, 1986.

\bibitem{vanlint}
J.~H. Van Lint, A.~Schrijver.
\newblock {\em Construction of strongly regular graphs, two-weight codes and partial geometries by finite fields.}
\newblock Combinatorica, 1:63--73, 1981.

\bibitem{polynorm}
E.~M. Luks, T.~Miyazaki.
\newblock {\em Polynomial-time normalizers.}
\newblock Discrete Mathematics and Theoretical Computer Science, 13(4):61--96, 2011.

\bibitem{muzychukOneDim}
M.~Muzychuk.
\newblock {\em A classification of one dimensional affine rank three graphs}, 2020,
{\tt arXiv:2005.05563 [math.CO]}.

\bibitem{peisert}
W.~Peisert.
\newblock {\em All Self-Complementary Symmetric Graphs.}
\newblock Journal of Algebra, 240:209--229, 2001.

\bibitem{praegerClosure}
C.~E. Praeger, J. Saxl.
\newblock {\em Closures of finite primitive permutation groups.}
\newblock Bulletin of the London Mathematical Society, 24(3):251--258, 1992.

\bibitem{praegerDecomp}
C.~E. Praeger, C. Schneider.
\newblock {\em Permutation Groups and Cartesian Decompositions}.
\newblock London Mathematical Society Lecture Note Series. Cambridge University Press, 2018.

\bibitem{schroder}
E.~M. Schr\"oder.
\newblock {\em Ein einfacher beweis des satzes von Alexandroff-Lester.}
\newblock Journal of Geometry, 37:153--158, 1990.

\bibitem{skresanovCex}
S.~V. Skresanov.
\newblock {\em Counterexamples to two conjuectures from the kourovka notebook.}
\newblock Algebra Logic, 58:249--253, 2019.

\bibitem{minperm}
A.~V. Vasilyev.
\newblock {\em Minimal permutation representations of finite simple exceptional groups of types \( E_6 \), \( E_7 \) and \( E_8 \).}
\newblock Algebra Logic, 36:302--310, 1997.

\bibitem{wilson}
R.~Wilson.
\newblock {\em The Finite Simple Groups}.
\newblock Graduate Texts in Mathematics. Springer London, 2009.

\end{thebibliography}
\end{document}